\documentclass{amsart}
\usepackage{amssymb,amsmath,amsthm,amsfonts}
\usepackage {latexsym,enumerate}
\usepackage{bbm}

\allowdisplaybreaks
\newcommand{\nc}{\newcommand}
\nc{\les}{\lesssim}
\nc{\nit}{\noindent}
\nc{\nn}{\nonumber}
\nc{\D}{\partial}
\nc{\diff}[2]{\frac{d #1}{d #2}}
\nc{\diffn}[3]{\frac{d^{#3} #1}{d {#2}^{#3}}}
\nc{\pdiff}[2]{\frac{\partial #1}{\partial #2}}
\nc{\pdiffn}[3]{\frac{\partial^{#3} #1}{\partial{#2}^{#3}}}
\nc{\abs}[1] {\lvert #1 \rvert}
\nc{\cAc}{{\cal A}_c}
\nc{\cE}{{\cal E}}
\nc{\cF}{{\cal F}}
\nc{\cP}{{\cal P}}
\nc{\cV}{{\cal V}}
\nc{\cQ}{{\cal Q}}
\nc{\cGin}{{\cal G}_{\rm in}}
\nc{\cGout}{{\cal G}_{\rm out}}
\nc{\cO}{{\cal O}}
\nc{\Lav}{{\cal L}_{\rm av}}
\nc{\cL}{{\cal L}}
\nc{\cB}{{\cal B}}
\nc{\cZ}{{\cal Z}}
\nc{\cR}{{\cal R}}
\nc{\cT}{{\cal T}}
\nc{\cY}{{\cal Y}}
\nc{\cX}{{\cal X}}
\nc{\cXT}{{{\cal X}(T)}}
\nc{\cBT}{{{\cal B}(T)}}
\nc{\vD}{{\vec \mathcal{D}}}
\nc{\efield}{\mathcal{E}}
\nc{\vE}{{\vec \efield}}
\nc{\vB}{{\vec \mathcal{B}}}
\nc{\vH}{{\vec \mathcal{H}}}
\nc{\mR}{\mathcal R}
\nc{\mF}{\mathcal F}
\nc{\ty}{{\tilde y}}
\nc{\tu}{{\tilde u}}
\nc{\tV}{{\tilde V}}
\nc{\Pc}{{\bf P_c}}
\nc{\bx}{{\bf x}}
\nc{\bX}{{\bf X}}
\nc{\bXYZ}{{\bf XYZ}}
\nc{\bY}{{\bf Y}}
\nc{\bF}{{\bf F}}
\nc{\bS}{{\bf S}}
\nc{\dV}{{\delta V}}
\nc{\dE}{{\delta E}}
\nc{\TT}{{\Theta}}
\nc{\dPsi}{{\delta\Psi}}
\nc{\order}{{\cal O}}
\nc{\Rout}{R_{\rm out}}
\nc{\eplus}{e_+}
\nc{\eminus}{e_-}
\nc{\epm}{e_\pm}
\nc{\sgn}{\text{sgn}}
\nc{\eps}{\varepsilon}
\nc{\vnabla}{{\vec\nabla}}
\nc{\G}{\Gamma}
\nc{\w}{\omega}
\nc{\mh}{h}
\nc{\mg}{g}
\nc{\vphi}{\varphi}
\nc{\tlambda}{\tilde\lambda}
\nc{\be}{\begin{equation}}
\nc{\ee}{\end{equation}}
\nc{\ba}{\begin{eqnarray}}
\nc{\ea}{\end{eqnarray}}

\nc{\g}{\gamma}
\nc{\ol}{\overline}

\newtheorem{theorem}{Theorem}[section]
\newtheorem{lemma}[theorem]{Lemma}
\newtheorem{prop}[theorem]{Proposition}
\newtheorem{corollary}[theorem]{Corollary}

\nc{\pT}{\partial_T}
\nc{\pz}{\partial_z}
\nc{\pt}{\partial_t}
\nc{\la}{\langle}
\nc{\ra}{\rangle}
\nc{\infint}{\int_{-\infty}^{\infty}}
\nc{\halfwidth}{6.5cm}
\nc{\figwidth}{10cm}
\newcommand{\f}{\frac}

\nc{\nlayers}{L} \nc{\nsectors}{M}
\nc{\indicator}{\mathbf{1}}
\nc{\Rhole}{R_{\rm hole}}
\nc{\Rring}{R_{\rm ring}}
\nc{\neff}{n_{\rm eff}}
\nc{\Frem}{F_{\rm rem}}
\nc{\R}{\mathbb R}
\nc{\mJ}{\mathcal J}
\nc{\C}{\mathbb C}
\nc{\Z}{\mathbb Z}
\nc{\N}{\mathbb N}
\nc{\DD}{\Delta}
\nc{\cD}{\mathcal D}
\nc{\lnorm}{\left\|}
\nc{\rnorm}{\right\|}
\nc{\rnormp}{\right\|_{\ell^{p,\eps}}}
\nc{\rar}{\rightarrow} 
\begin{document}
	
	\begin{abstract}
		
		We consider the higher order Schr\"odinger operator $H=(-\Delta)^m+V(x)$ in $n$ dimensions with real-valued potential $V$ when $n>4m-1$, $m\in \mathbb N$.  We show that for any $\frac{2n}{n-4m+1}<p\leq \infty$ and $0\leq \alpha <\frac{n+1}{2}-2m-\frac{n}p$, there exists a real-valued, compactly supported potential $V\in C^{\alpha}(\R^n)$  for which the wave operators $W^{\pm}$ are not bounded on $L^p(\R^n)$.  As a consequence of our analysis we show that the wave operators for the usual second order Schr\"odinger operator $-\Delta+V$ are unbounded on $L^p(\R^n)$ for $n>3$ and $\frac{2n}{n-3}<p\leq \infty$ for insufficiently differentiable potentials $V$, and show a failure of   $L^{p'}\to L^p$ dispersive estimates that may be of independent interest.

	\end{abstract}

	\title[Wave operators for higher order  Schr\"odinger operators]{\textit{Counterexamples to $L^p$ boundedness of wave operators for classical and higher order Schr\"odinger operators } } 
	
	\author[M.~B. Erdo\smash{\u{g}}an, W.~R. Green]{M. Burak Erdo\smash{\u{g}}an, Michael Goldberg and William~R. Green}
	\thanks{  The first author was partially supported by the NSF grant  DMS-2154031 and Simons Foundation Grant 634269.  The second author is partially supported by Simons Foundation
		Grant 635369. The third author is partially supported by Simons Foundation
Grant 511825. }
	\address{Department of Mathematics \\
		University of Illinois \\
		Urbana, IL 61801, U.S.A.}
	\email{berdogan@illinois.edu}
	\address{Department of Mathematics\\
		University of Cincinnati \\
		Cincinnati, OH 45221 U.S.A.}
	\email{goldbeml@ucmail.uc.edu}
	\address{Department of Mathematics\\
		Rose-Hulman Institute of Technology \\
		Terre Haute, IN 47803, U.S.A.}
	\email{green@rose-hulman.edu}

	\maketitle

	\section{Introduction}
	
	We consider the higher order Schr\"odinger equation 
	\begin{align*}
	i\psi_t =(-\Delta)^m\psi +V\psi, \qquad x\in \R^n,  \quad  m\in \mathbb N,
	\end{align*}
	with a real-valued and decaying potential  $V$.  We denote the free higher order Schr\"odinger operator by $H_0=(-\Delta)^m$ and the perturbed operator by $H=(-\Delta)^m+V(x)$. We study the  wave operators, which are defined by
	$$
	W_{\pm}=s\text{\ --}\lim_{t\to \pm \infty} e^{itH}e^{-itH_0}.
	$$
	Recent work by the first and third authors, \cite{EGWaveOp}, showed that for $m>1$ and $n > 2m$ the wave operators extend to bounded operators on $L^p(\R^n)$ for $1\leq p\leq \infty$ for sufficiently smooth small potentials. The case $m=1$ was established by Yajima in \cite{YajWkp1}. Here we show that when $n>4m-1$ and $\frac{2n}{n-4m+1}<p\leq \infty$ the $L^p$ boundedness of the wave operators may fail even for compactly supported continuous potentials if the potential is not sufficiently smooth.   Our main result is the following.
	
	\begin{theorem}\label{thm:main}
		
		Fix $m\in \mathbb N$,  let $n>4m-1$ and $\frac{2n}{n-4m+1}<p\leq \infty$, for all $0\leq \alpha<\frac{n+1}{2}-2m-\frac{n}{p}$  there exists a real-valued compactly supported potential of class $C^{\alpha}(\R^n)$  for which the wave operators $W_{\pm}$ are not bounded on $L^{p}(\R^n)$.		
		
	\end{theorem}

	For the convenience of the reader, in Section~\ref{sec:fail} we prove Theorem~\ref{thm:main} when $p=\infty$, then adapt the argument in Section~\ref{sec:ext to p} to the remaining cases of $\frac{2n}{n-4m+1}<p< \infty$.  By considering frequency-localized dispersive estimates, we provide a direct argument for arbitrary integer order $m$ that simplifies the $m=1$ argument for the dispersive bounds in \cite{GV} and extends to $m\geq 1$.  
	
	For comparison, the results on boundedness in \cite{EGWaveOp} require some smoothness on the potential when $n\geq 4m-1$.  Writing $\la x\ra$ to denote $  (1+|x|^2)^{\f12}$, $\mathcal F(f)$ or $\widehat f$ to denote the Fourier transform of $f$ and defining the norm $\|f\|_{H^{\delta}}=\|\la \cdot \ra^{\delta} \widehat f(\cdot)\|_2$, we recall the relevant statements  below.  

	\begin{theorem}[Theorem~1.1 in \cite{EGWaveOp}]
		Let $n>2m$. 
		Assume that the $V$ is a real-valued potential on $\R^n$ and   fix $0<\delta\ll 1$. Then $\exists C=C(\delta,n,m)>0 $ so that the wave operators  extend   to   bounded operators on $L^p(\R^n)$ for all $1\leq p\leq \infty$, provided that 
		\begin{enumerate}[i)]
			
			\item $\big\| \la \cdot \ra^{\frac{4m+1-n}{2}+\delta} V(\cdot)\big\|_{2}<C$   when $2m<n<4m-1$,
			
			\item $\big\|\la \cdot \ra^{1+\delta}V(\cdot)\big\|_{H^{\delta}}<C$ when $n=4m-1$,
			
			\item  $\big\|\mathcal F(\la \cdot \ra^{\sigma} V(\cdot))\big\|_{L^{ \frac{n-1-\delta}{n-2m-\delta} }}<C$ for some $\sigma>\frac{2n-4m}{n-1-\delta}+\delta$   when $n>4m-1$.
			
		\end{enumerate}
		
	\end{theorem}
	See \cite{EGWaveOp} for other statements that remove the smallness requriement on the potential. 	The arbitrarily small differentiability assumption when $n=4m-1$ appears to be an artifact of the method in \cite{EGWaveOp}.
	We note that the norm used when $n>4m-1$ is finite when $\la x\ra^{\sigma}V(x)$ has more than $\frac{n}{n-1}(\frac{n+1}{2}-2m)$    derivatives in $L^2(\R^n)$.  We   believe this requirement on smoothness of the potential is not sharp, in light of the counterexample constructed here.  When $m=1$ the counterexample in \cite{GV} and the positive result of the first and third authors in \cite{EG} show that $\alpha=\frac{n-3}{2}$ is sharp for $L^1\to L^\infty$ dispersive estimates, at least when $n=5,7$.  In effect, our result shows that one cannot expect decay of a class of potentials alone suffice to ensure the boundedness of the wave operators as is the case in lower dimensions, when $2m<n<4m-1$ above.  Instead, one also expects a degree of differentiability on the potential is needed to ensure boundedness in high dimensions.

	The potentials considered here also suffice to imply, see for example \cite{Sche,agmon,ScheArb}, the existence, $L^2$-boundedness, asymptotic completeness, and intertwining identity for the wave operators.
	In particular, we have the identity
	\begin{align}\label{eq:intertwining}
	f(H)P_{ac}(H)=W_\pm f((-\Delta)^m)W_{\pm}^*.
	\end{align}
	Here $P_{ac}(H)$ is the projection onto the absolutely continuous spectral subspace of $H$, and $f$ is any Borel function.  One use of \eqref{eq:intertwining} is to obtain $L^p$-based mapping properties for the perturbed operator $f(H)P_{ac}(H)$ from the simpler free operator $f((-\Delta)^m)$.  The boundedness of the wave operators on $L^p(\R^n)$ for any choice of $p\geq 2$ with the function $f(\cdot)=e^{-it(\cdot)}$ yield dispersive estimates of the form
	\begin{align}
	\|e^{-itH}P_{ac}(H)\|_{L^{p'}\to L^p}\les |t|^{-\frac{n}{2m}+\frac{n}{pm} }, \qquad \frac{1}{p}+\frac{1}{p'}=1.
	\end{align}
	To establish our results, we appeal to the intertwining identity and show that certain dispersive estimates fail for a specific choice of $f(\cdot)$ for small time when localized to large frequencies.
	
	We note that our results include the case of $m=1$, the usual second order Schr\"odinger operator in dimensions $n>3$.  The results of Theorem~\ref{thm:main} are, to the best of the authors' knowledge, new even in this case. In particular showing that the wave operators need not be bounded on $\frac{2n}{n-3}<p\leq \infty$ for insufficiently smooth potentials.  We also show that the  $L^p\to L^{p'}$ dispersive estimates 
	$$
	\bigg\| e^{it(-\Delta+V)}P_{ac}(-\Delta+V)\bigg\|_{p\to p'} \les |t|^{-\frac{n}{2}(\frac2p-1)}, \quad \text{when } n>3 \text{ and }  1\leq p<\frac{2n}{n+3}
	$$
	fail in this case, see Corollary~\ref{cor:m=1} below.

	Our work is inspired by recent work by the first and third authors on the $L^p$-continuity of the higher order wave operators in \cite{EGWaveOp} and the counterexample to dispersive estimates in the $m=1$ case of the second author and Visan in \cite{GV}.  The recent work on the $L^p$-boundedness of the wave operators for higher order Schr\"odinger operators was sparked by the work of Feng, Soffer, Wu and Yao on weighted $L^2$-based ``local dispersive estimates" for higher order Schr\"odinger operators considered in \cite{soffernew}, which extends the $m=1$ result of Jensen \cite{Jen}. In addition, the recent work on the $L^p(\R^3)$ boundedness of the wave operators for the fourth order ($m=2$) Schr\"odinger operators by the second and third authors \cite{GG4wave}, and the extensive works in the case of $m=1$ most notably that of Yajima, \cite{YajWkp1,YajWkp2,YajWkp3}.  The $L^p(\R)$ boundedness has recently been investigated by  Mizutani, Wan and Yao  in \cite{MWY}. 
	The $L^2$ existence and other properties of the higher order wave operators have been studied by many authors, for example by Agmon \cite{agmon}, Kuroda \cite{Kur1, Kur2}, H\"ormander \cite{Hor}, and Schechter, \cite{Sche,ScheArb}.  
	

	Similar to the usual second order Schr\"odinger operator, there is a Weyl criterion and $\sigma_{ac}(H)=\sigma_{ac}(H_0)=[0,\infty)$ for sufficiently decaying potentials.  In contrast, decay of the potential is not sufficient to ensure the lack of eigenvalues embedded in the continuous spectrum for the higher order operators, \cite{soffernew}.  Even perturbing with compactly supported, smooth potentials may induce embedded eigenvalues.  For the potentials we consider, we show that the eigenvalues cannot be too large, which allows the use of a limiting absorption principle for the perturbed operator on the interval $[C_V,\infty)$ for a sufficiently large $C_V>0$, see Lemma~\ref{lem:eval} below.
	
	To prove Theorem~\ref{thm:main} we show the failure of a dispersive estimate based on detailed analysis of oscillatory integrals involving the resolvent operators.  The splitting identity for $z\in\C\setminus[0,\infty)$, (c.f. \cite{soffernew}) allows us to study the resolvent $((-\Delta)^m-\lambda)^{-1}$ in terms of $R_0(z)=(-\Delta-z)^{-1}$, the usual ($2^{nd}$ order) Schr\"odinger resolvent.
	\be\label{eqn:Resol}
	\mR_0(z)(x,y):=((-\Delta)^m -z)^{-1}(x,y)=\frac{1}{ mz^{1-\frac1m} }
	\sum_{\ell=0}^{m-1} \omega_\ell R_0 ( \omega_\ell z^{\frac1m})(x,y)
	\ee
	here $\omega_\ell=\exp(i2\pi \ell/m)$ are the $m^{th}$ roots of unity.  
	It is convenient to use the change of variables $z= \lambda^{2m}$ with $\lambda$ restricted to the sector in the complex plane with $0<\arg(\lambda)<\pi/m$,
	\be\label{eqn:Resolv}
	\mR_0( \lambda^{2m})(x,y):=((-\Delta)^m -\lambda^{2m})^{-1}(x,y)=\frac{1}{ m\lambda^{2m-2}}
	\sum_{\ell=0}^{m-1} \omega_\ell R_0 ( \omega_\ell  \lambda^2)(x,y).
	\ee

	The kernal of higher order Schr\"odinger resolvents  $\mR_0(\lambda^{2m})$ grow like $\lambda^{\frac{n+1}{2}-2m}$ in the spectral variable when $n> 4m-1$, this growth necessitates a control over derivatives of the potential which was controlled in terms of $\mathcal F L^r$ norms in \cite{EGWaveOp}, similar to the conditions for the second order Schr\"odinger established by Yajima, \cite{YajWkp1}. When $m=1$ this growth may be exploited to cause a failure of $L^1\to L^\infty$ dispersive estimates in dimensions greater than three without sufficient smoothness of the potential; a counterexample was constructed by the second author and Visan, \cite{GV}.   We adapt and simplify this argument by considering $L^1\to L^\infty$ estimates of the operators of the form $H^{\frac{n(m-1)}{2m}}e^{itH}\psi(H/L^{2m})P_{ac}(H)$, with $\psi(s)$ is a cut-off to frequencies of size $s\approx 1$ and $L$ is a sufficiently large constant.  Treating this operator as an element of the functional calculus, the modified Stone's formula is
	$$
		H^{\frac{n(m-1)}{2m}}e^{itH}\psi(H/L^{2m})P_{ac}(H)f(x)=\frac{1}{2\pi i} \int_0^{\infty} \lambda^{\frac{n(m-1)}{2m}}e^{it\lambda }\psi(\lambda/L^{2m})\, dE_{ac}(\lambda)	f(x).
	$$
	Here the difference of the limiting resolvent operators, $\mR_V^\pm (\lambda)=\lim_{\epsilon\to 0^+} ((-\Delta)^m+V-\lambda \mp i \epsilon)^{-1}$,  provides the spectral measure $dE_{ac}(\lambda)=[\mR_V^+(\lambda)-\mR_V^-(\lambda)]\, d\lambda$.  These operators are well-defined between weighted spaces by Agmon's limiting absorption principle, and are well studied in \cite{EGWaveOp}.
	By appropriately relating the frequency and time, we show a dispersive estimate that the free operator (when $V=0$) satisfies cannot hold for the perturbed operator as $t\to 0$.  Using the intertwining identity, \eqref{eq:intertwining}, we show the wave operators are not bounded on $L^\infty$.  By appropriately rescaling both the powers of $H$ and the time decay, we extend the argument to $p<\infty$.
	
	The presence of an operator $H^{\frac{n(m-1)}{2m}}$ in the dispersive estimate is natural in the following sense: For $m>1$, the free fundamental solution $e^{itH_0}(x,y)$ has a central peak where $|x-y| \les |t|^{\frac{1}{2m}}$, then experiences a combination of polynomial decay and oscillation (similar to an Airy function) for larger distances $|x-y|$.  Differentiating $n(m-1)$ times yields an operator whose kernel is approximately the same size for all values of $|x-y|$, so the $L^1 \to L^\infty$ bound can be achieved at every length scale or when localized to any frequency band.

	The paper is organized as follows.  In Section~\ref{sec:free} we show that the free operator satisfies a family of dispersive estimates. In Section~\ref{sec:fail} we show that the perturbed operator cannot satisfy the same $L^1\to L^\infty$ dispersive estimate as the free operator for a specifically constructed potential.  As a consequence, the wave operator is not bounded on $L^\infty$.  Finally, in Section~\ref{sec:ext to p} we show how the argument may be adapted to construct a class of potentials for which the wave operators are unbounded on the larger range of $\frac{2n}{n-4m+1}<p\leq \infty$.

	\section{The free estimate}\label{sec:free}
	
	In this section we establish a family of dispersive estimates that the free operator $H_0=(-\Delta)^m$ satisfies.   In Section~\ref{sec:fail} we construct a compactly supported potential for which the perturbed operator cannot satisfy the analogous bound.
	
	\begin{prop}\label{prop:free}
		
		$$\big\| H_0^{\frac{n(m-1)}{2m}}e^{itH_0}\big\|_{1\to \infty} \les |t|^{-\frac{n}{2}},$$
and for all $\sigma \in \R$,
		$$\big\| H_0^{\frac{n(m-1)+i\sigma}{2m}}e^{itH_0}\big\|_{1\to \infty} \les (1+|\sigma|)^{\frac{n+2}{2}}|t|^{-\frac{n}{2}}.$$
		
	\end{prop}
	\begin{proof}
		Using the splitting identity \eqref{eqn:Resolv}, along with the fact that $R_0^+(\omega_\ell \lambda^2)=R_0^-(\omega_\ell \lambda^2)$ for $\omega_\ell \notin \R$, that is when $\omega_\ell \neq 1$, we have
	\begin{align}
		[\mR_0^+-\mR_0^-](\lambda^{2m})(x,y)=\frac{1}{m\lambda^{2m-2}} [R_0^+-R_0^-](\lambda^2)(x,y).
	\end{align}
	Hence, one need only understand the usual second order free resolvent to understand the estimates in Propostion~\ref{prop:free}.  As usual, we estimate the evolution by using the functional calculus, in this case, we'll estimate (ignoring constants)
	\begin{align}\label{eqn:free disp}
		\sup_{x,y\in \R^n} \bigg| \int_0^\infty e^{it\lambda^{2m}} \lambda^{n(m-1)+1 + i\sigma} [R_0^+-R_0^-](\lambda^2)(x,y)\, d\lambda \bigg|.
	\end{align}
We utilize the ``symbol class" representation for the Schr\"odinger operators used by the second author and Visan in \cite{GV} to write
	$$
		[R_0^+-R_0^-](\lambda^2)(x,y)= \lambda^{n-2}\big(e^{i\lambda |x-y|} A_1(\lambda |x-y|) - e^{- i\lambda |x-y|} A_2(\lambda |x-y|)\big),
	$$
	where $A_1(s), A_2(s)=\widetilde O(\la s\ra^{\frac{1-n}{2}} )$.  Here $f(s)=\widetilde O(\la s\ra^{j} )$ means $|\partial_s^k f(s)|\les \la s\ra^{j-k}$ for each $k=0,1,2,\dots$

We  only consider the case $t>0$ and when there is a `$-$' sign on the phase; the other cases are similar or easier.
Therefore, we  only prove that   
	$$
		\sup_{x,y\in \R^n,L>0,t>0} t^{n/2}\bigg| \int_0^\infty e^{it\lambda^{2m}-i\lambda |x-y|} \lambda^{nm-1+i\sigma} A_2(\lambda |x-y|) \chi(\lambda/L)\, d\lambda \bigg| \les 1.
	$$
	Only the function $A_2(\lambda|x-y|)$ is present, so we omit the subscript in the rest of the calculation.  Changing the variable $\lambda |x-y| \mapsto \lambda$ and appropriately renaming the variables $t\mapsto t|x-y|^{2m}$, $L|x-y|\mapsto L$, it suffices to prove that
\begin{align}\label{eqn:free dis1}
		\sup_{t>0,L>0} t^{n/2}\bigg| \int_0^\infty e^{it\lambda^{2m}-i\lambda} \lambda^{nm-1+i\sigma} A(\lambda) \chi(\lambda/L)\, d\lambda \bigg| \les 1.
\end{align}
Note that the phase $\phi(\lambda)= t\lambda^{2m}-\lambda$ has a critical point at $\lambda_0=(2mt)^{-\frac1{2m-1}}$. Let $\psi_1$ be a smooth cutoff for the set $|\lambda|\approx \lambda_0$, $\psi_2$ for the set $|\lambda|\ll \lambda_0$ and $\psi_3$ for the set $|\lambda|\gg \lambda_0$ so that $\psi_1+\psi_2+\psi_3\equiv 1$.  
Let 
$$
I_j:=\int_0^\infty e^{it\lambda^{2m}-i\lambda} \lambda^{nm-1+i\sigma} A(\lambda)  \psi_j(\lambda)\chi(\lambda/L)\, d\lambda,\,\,\,\,j=1,2,3.
$$
Note that in the support of   $\psi_1(\lambda)$, we have 
$\phi^{\prime\prime}(\lambda)\approx t\lambda_0^{2m-2}$. Therefore by the corollary on page 334 of \cite{stein}, we have 
\begin{multline*}
|I_1|\les (t\lambda_0^{2m-2})^{-1/2}\int \big|\partial_\lambda\big(\lambda^{nm-1+i\sigma} A(\lambda)  \psi_1(\lambda)\chi(\lambda/L)\big)\big|d\lambda\\
\les t^{-1/2}\la \sigma\ra \lambda_0^{1-m+nm-1}\la \lambda_0\ra^{\frac{1-n}2}\les t^{-1/2}\la \sigma \ra \lambda_0^{(2m-1)(n-1)/2}\approx \la \sigma \ra t^{-n/2}.
\end{multline*}
Note that on the support of $\psi_2$, where $|\lambda|\ll \lambda_0$, we have  $|\phi^\prime|\gtrsim 1$. 
Therefore we may integrate by parts $N\leq nm-1$ times  to obtain the bound
$$
|I_2|\les \int_0^\infty \Big|\Big[\partial_\lambda\frac1{\phi^\prime(\lambda)}\Big]^{N} \big(\lambda^{nm-1+i\sigma} A(\lambda)  \psi_2(\lambda)\chi(\lambda/L)\big)\Big|d\lambda.
$$
Noting that the effect of a derivative on each term can be bounded by multiplication by $\la\sigma\ra \lambda^{-1}$, we have
$$
|I_2|\les \la \sigma\ra^N \int_0^{\lambda_0} \lambda^{nm-1-N} d\lambda \les \la \sigma\ra^N \lambda_0^{nm-N} \les \la \sigma\ra^N t^{-\frac{nm-N}{2m-1}},\,\,\,0\leq N\leq nm-1.
$$
Picking $N=0$ when $t>1$ and $\frac{n}{2} \leq N \leq nm-1$ for $0<t<1$ yields stronger bounds than $t^{-n/2}$.   

It remains to consider $I_3$. On the support of $\psi_3$, we have $|\phi^\prime|\gtrsim t\lambda^{2m-1}\gtrsim 1$. Therefore by integration by parts $N$ times we obtain
$$
|I_3|\les \la\sigma\ra^N \int_{\lambda_0}^\infty \lambda^{nm-1-N} \la \lambda\ra^{\frac{1-n}2} [t\lambda^{2m-1}]^{-N} d\lambda=
\la\sigma\ra^N t^{-N}\int_{\lambda_0}^\infty \lambda^{nm-1-2mN} \la \lambda\ra^{\frac{1-n}2} d\lambda. 
$$
When $t<1$, $\lambda_0\gtrsim 1$ and taking $N>\frac{nm+\frac12-\frac{n}2}{2m}$ implies the stronger bound $\la\sigma\ra^N t^{-\frac{n}2}t^{\frac{N-\frac12}{2m-1}}$.
For $t>1$, take $N=0$ on $[\lambda_0,t^{-\frac1{2m}}]$, and  take $N>n/2$ on $[t^{-\frac1{2m}},\infty)$  to obtain
$$
\les \int_0^{t^{-\frac1{2m}}} \lambda^{nm-1}d\lambda +\la\sigma\ra^N t^{-N}\int_{t^{-\frac1{2m}}}^\infty \lambda^{nm-1-2mN}  d\lambda\les t^{-n/2}+\la\sigma\ra^N t^{-N} t^{-\frac1{2m}(nm-2mN)}\les \la\sigma\ra^Nt^{-n/2}. 
$$
At no point was integration by parts needed more than $\frac{n}{2}+1$ times, so the sum $I_1 + I_2 + I_3$ is bounded
by $\la\sigma\ra^{\frac{n+2}{2}} t^{-n/2}$. 
\end{proof}

The above bound is all we need in Section~\ref{sec:fail} to establish the result of Theorem~\ref{thm:main} when $p=\infty$.  For the full range of $p$ considered in Section~\ref{sec:ext to p}, we need the following corollary.
\begin{corollary}\label{cor:free} Let $\psi$ be a smooth cutoff for the set $\{\lambda\in \R: \lambda\approx 1\}$. 
For all $p\in [1,2]$ and $L>0$ we have		
		$$\| H_0^{\frac{n(m-1)}{2m}(\frac2p-1)}e^{itH_0}\psi(H_0/L)\|_{p\to p^\prime} \les |t|^{-\frac{n}p+\frac{n}{2}},$$
		where the implict constant is independent of $p$ and $L$. 
 \end{corollary}
\begin{proof} The claim without $\psi$ follows from complex interpolation of the bound in Proposition~\ref{prop:free} with the $L^2$ conservation law. The claim with $\psi$ follows from this and the fact that $\widehat\psi\in L^1$.
\end{proof}

\section{Failure of the dispersive estimate}\label{sec:fail}

In this section
we show that the perturbed evolution cannot satisfy the same dispersive bound as the free evolution.  For clarity, we concentrate in this section on the failure of the $L^1\to L^\infty$ dispersive estimate and consequently the unboundedness of the wave operator on $L^\infty$.  In Section~\ref{sec:ext to p}, we show how to adapt this argument to a larger range of $p$.

Let $\psi$ be a smooth cutoff for frequencies $\lambda\approx 1$. 
Note that the $L^\infty$ boundedness of wave operators together with asymptotic completeness, the intertwining identity \eqref{eq:intertwining} and the bound in Corollary~\ref{cor:free}  imply that  the bound 
\be\label{Hpsi}
\big\|H^{\frac{n(m-1)}{2m}}e^{itH}\psi(H/L^{2m})P_{ac}(H)\big\|_{L^1\to L^\infty}\les |t|^{-n/2}
\ee
holds for all $L>0$ uniformly in $L$.  

Let $C^\alpha(B(0,2))$ be the Banach space of real-valued $C^\alpha$ functions supported in $B(0,2)$. We  prove that the inequality  
\be\label{supHpsi}
\sup_{t>0, L>0} t^{n/2}\big\|H^{\frac{n(m-1)}{2m}}e^{itH}\psi(H/L^{2m})P_{ac}(H)\big\|_{L^1\to L^\infty}\leq C_V 
\ee
cannot hold for all $V\in C^\alpha(B(0,2))$ if $0\leq\alpha<\tfrac{n+1}2-2m$.  By showing that this dispersive bound fails for insufficiently smooth potentials $V$ for small times $t\to 0$, we show that the wave operators are not bounded on $L^\infty(\R^n)$.  

To do so we iterate the resolvent identity to expand the perturbed resolvent into a Born series
\begin{align}\label{eqn:born identity}
	\mR_V(z)=\sum_{j=0}^{2M-1}\big[ \mR_0(z)(-V\mR_0(z))^j \big]- (\mR_0(z)V)^M \mR_V(z) (V\mR_0(z))^M.
\end{align}
We note that the $j=0$ term corresponds to the free evolution.
In the following subsections we first show that the first term of the Born series, when $j=1$, for large frequencies and small times does not satisfy the dispersive bound.  In the following subsection, we show that the final term in the identity (the ``tail") obeys the dispersive bound.  Finally, we use these facts to show that the full evolution can't satisfy the dispersive bound.

\subsection{The first term of the Born series}\label{sec:first}
First we   consider the first term of the Born series  of the operator in \eqref{supHpsi}. Ignoring the constants,   
we write it's kernel, $K_{L,t}$, as a difference of kernels
\begin{align*}
K^{\pm}_{L,t}(x,y)= \int_0^\infty \int_{\R^n} e^{it\lambda^{2m}} \lambda^{2m-1} \lambda^{n(m-1)} \mR_0^\pm(\lambda^{2m})(x,z) V(z) 
\mR_0^\pm(\lambda^{2m})(z,y) \psi(\lambda^{2m}/L^{2m}) dz d\lambda 
\end{align*} 
Let $f$ be the $L^1$ normalized characteristic function of $B(0,1)$, and let $f_L(x)=(CL)^{n}f(x(CL))$. Below $C$ will be chosen to be large to guarantee that the relevant values of $|x|$ and $|y|$ are $\ll1$.  In what follows we prove that the dispersive estimate fails by letting $t=L^{-(2m-1)}$ and taking $L\to \infty$.  With $t^{\frac{n}2}=L^{\frac{n}2-nm}$, we define
$$a_{1,L}(V)= L^{\frac{n}2-nm} \int_{\R^n\times\R^n} K_{L,L^{-(2m-1)}}(x,y) f_L(x)f_L(y) dx dy.
$$
Note that  $a_{1,L}$ is a linear operator on $C^\alpha(B(0,2))$. By uniform boundedness principle, if the bound in \eqref{supHpsi} holds for the first term of the Born series and for all $V\in C^\alpha(B(0,2))$, then we have 
\be\label{UBBalpha}
\forall V\in C^\alpha(B(0,2)),\,\,\,\,
\sup_{L>0} |a_{1,L}(V)|\leq C_\alpha \|V\|_{C^\alpha(B(0,2))}.
\ee  
Therefore, it suffices  to find a sequence $\{V_L\}_{L>1} \subset C^\alpha(B(0,2))$  so that 
$$
\lim_{L\to\infty}\frac{|a_{1,L}(V_L)|}{\|V_L\|_{C^\alpha(B(0,2))}}=\infty.
$$
 
Note that for $\lambda |x-z| >1$, we have
\be\label{schR0}
R_0^\pm(\lambda^2)(x,z)=e^{\pm i\lambda |x-z|} \lambda^{n-2} \omega^\pm(\lambda |x-z|),
\ee
where $\omega^\pm(s)=c_\pm s^{\frac{1-n}{2}} +\widetilde{O}(s^{-\frac{1+n}2}),$  $s>1$.  In our argument below, we have $\lambda \approx L$ large and $|x-z|\approx 1$, which allows us to avoid the logarithmic behavior of even dimensional resolvents when $\lambda |x-z|\ll 1$. Using this in the splitting formula for $\mR_0^\pm(\lambda^{2m})$, \eqref{eqn:Resolv}, and noting the exponential decay of the remaining terms in the splitting identity allows them to be absorbed into $\omega^\pm(\lambda|x-z|)$, we have for $\lambda|x-z|>1$
$$
\mR_0^\pm(\lambda^{2m})(x,z)=e^{\pm i\lambda |x-z|} \lambda^{n-2m} \omega^\pm(\lambda |x-z|).
$$
Here, by a slight abuse of notation, $\omega^\pm$ satisfies the same bounds as $\omega^\pm$ in \eqref{schR0}.  One can also see Lemmas~3.2 and 6.2 in \cite{EGWaveOp} for more detailed representations of the resolvent.  Using this in the formula for $K^\pm_{L,t}(x,y)$, we have   
\begin{multline*}
K^\pm_{L}(x,y):=K^\pm_{L,L^{-(2m-1)}}\\ = \int_0^\infty \int_{\R^n} e^{i\frac{\lambda^{2m}}{L^{ 2m-1 }}\pm i\lambda R}\lambda^{mn+n-2m-1}\omega^\pm(\lambda r)\omega^\pm(\lambda s) V(z) \psi(\lambda^{2m}/L^{2m}) dz d\lambda
\end{multline*}
where $R:=r+s:=|x-z|+|z-y|$.  
Letting $\varphi(\lambda)= \lambda^{mn+n-2m-1} \psi(\lambda^{2m})$ be a modified cutoff to the interval $\lambda \approx 1$, we rewrite the kernel above as
$$
 L^{mn+n-2m-1} \int_0^\infty \int_{\R^n} e^{i\frac{\lambda^{2m}}{L^{ 2m-1 }}\pm i\lambda R}  \omega^\pm(\lambda r)\omega^\pm(\lambda s) V(z) \varphi(\lambda/L) dz d\lambda.
$$ 
In what follows we  have   $r,s=1+o(1)$ and $R=2+o(1)$, this is accomplished by taking $|x|,|y|$ small and $V(z)$ supported on a sufficiently small neighborhood of $|z|=1$. Changing the variable $\lambda\to L\lambda$ we have
$$
K^\pm_{L}(x,y) = 
 L^{mn+n-2m} \int_{\R^n} \int_0^\infty e^{iL(\lambda^{2m}\pm  \lambda R)}  \omega^\pm(\lambda Lr)\omega^\pm(\lambda L s) V(z) \varphi(\lambda) dz d\lambda,
$$
Note that the contribution of $+$ sign above is $O(L^{-N} \|V\|_{L^1})$ by nonstationary phase since $\varphi$ is smooth and supported on the set $\lambda\approx 1$ and $|\partial_\lambda^k \omega^{\pm}(\lambda Lr)|\leq C_k$ for all $k$. For the $-$ sign we have a critical point at $\lambda_0=(R/2m)^{\frac1{2m-1}}\approx 1$, and hence by stationary phase we have 
\begin{multline*}
	K^-_L(x,y)=C_m L^{mn+n-2m-\frac12} \int_{\R^n} \omega^-(\lambda_0 Lr)\omega^-(\lambda_0 Ls)\varphi(\lambda_0) e^{i c_m L  (\frac{R}{2})^{\frac{2m}{2m-1}}}V(z)dz\\
	+O(L^{mn-2m-\frac12}   \|V\|_{L^1}),
\end{multline*}
where  $c_m=(\tfrac1m)^{\frac{2m}{2m-1}}(1-2m)$.
For the error term, we used the bound $|\omega^\pm(\lambda)|\approx \la \lambda\ra^{\frac{1-n}2}.$ Using the asymptotic expansion above for $\omega^\pm$ and slightly modifying $\varphi$, we can rewrite this as
$$
K^-_L(x,y)= C_m L^{mn-2m+\frac12}  \int_{\R^n}  (rs)^{\frac{1-n}2} \varphi(\lambda_0) e^{i c_m L  (\frac{R}{2})^{\frac{2m}{2m-1}}}V(z)dz+O(L^{mn -2m-\frac12}   \|V\|_{L^1}),
$$

Let
 $$
 V_L(z) 
 = \cos\big(c_m L|z|^{\frac{2m}{2m-1}} \big) \rho_\delta(z),
 $$ where  $\delta\ll 1$, and $\rho_\delta$ is a smooth cutoff for the set $|z|\in (1-\delta,1+\delta)$. We  determine  $\delta$ later. However, $\delta\ll 1$ is fixed and we later let $L\to \infty$. Therefore, for sufficiently large $L$,
 $$\|V_L\|_{C^\alpha(B(0,2))}\approx \delta^{-\alpha}+ L^\alpha \approx L^{\alpha }
 $$
 For $x,y\in$supp$(f_L)$, i.e., $|x|,|y|\leq  \frac1{CL}$, we consider
\begin{multline*}
K_L(x,y)=K^+_L(x,y)-K^-_L(x,y)= O(L^{mn -2m-\frac12 }\delta)+\\
\frac{C_m}{2i} L^{mn-2m+\frac12 }  \int_{\R^n}  (rs)^{\frac{1-n}2} \varphi(\lambda_0) e^{i c_m L   [(\frac{R}{2})^{\frac{2m}{2m-1}}-|z|^{\frac{2m}{2m-1} } ]}\rho_\delta(z)dz\\
+
\frac{C_m}{2i} L^{mn-2m+\frac12 } \int_{\R^n} (rs)^{\frac{1-n}2} \varphi(\lambda_0) e^{i c_m L   [(\frac{R}{2})^{\frac{2m}{2m-1}}+|z|^{\frac{2m}{2m-1} } ]}\rho_\delta(z)dz.
\end{multline*}
Note that by nonstationary phase, for all $N\in\N$, the last integral is $O(L^{mn-2m+\frac12 -N} \delta^{1-N})$ since $|\nabla^k \rho_\delta|\les \delta^{-k}$. Using this with $N=1$, the error terms are combined to $O(L^{mn -2m-\frac12 })$. On the other hand, since in the support of $f_L$
$$\big|\frac{R}2-|z|\big|\leq\frac12 (|x|+|y|)\leq \frac1{CL},
$$
by choosing $C$ in the definition of $f_L$ sufficiently large depending only on $m$, we see that the phase in the first integral is $o(1)$.
Therefore,   we have the following lower bound for the first integral: 
$$L^{mn-2m+\frac12 }\delta.
$$
We can conclude that there is a constant $c\in \C$:
$$
\Re \big[c K_L^+(x,y)-c K_L^-(x,y)\big] \geq  L^{mn-2m+\frac12}\delta,
$$
for $\delta\ll 1$ fixed,  $L$ sufficiently large, and $|x|,|y|\leq \frac1{CL}$.  Therefore, we have 
$$\frac{|a_{1,L}(V_L)|}{\|V_L\|_{C^\alpha(B(0,2))}}\gtrsim L^{\frac{n+1}2 -2m-\alpha }\delta \to \infty
$$
 as $L\to\infty$ unless $\alpha\geq \frac{n+1}2 -2m$.

	\subsection{The Tail of the Born Series}\label{sec:Tail Bdd hi}
	
	Here we establish that, upon sufficient iteration of the Born series, that the tail of the Born series is bounded as $t\to 0$.
	Specifically, we show the following.

	\begin{prop}\label{prop:hi tail}
		
		Assume $V\in L^\infty(\R^n)$ is supported on $B(0,2)$, given any choice of $L>C_{V,n,m}$, where $0<C_{V,n,m}<\infty$ is a constant depending on the size of the potential, spatial dimension and order of the operator, there exists an $M$ so that the tail of the Born series satisfies the following bound
		\begin{align*}
			\sup_{x,y\in\R^n} \bigg|\int_0^\infty e^{it\lambda}\psi(\lambda/L^{2m}) \lambda^{\frac{n(m-1)}{2m}}(\mR_0^\pm(\lambda) V)^M \mR_V^\pm (\lambda) (V\mR_0^\pm(\lambda))^M(x,y) \, d\lambda \bigg| \les 1. 
		\end{align*}
		
	\end{prop}

	We utilize the limiting absorption principle established in \cite{soffernew}.  For fixed $m,n$, there are no positive eigenvalues of $H$ on the support of the cut-off provided we take $L$ sufficiently large, depending only on the size of $V$.   

\begin{lemma}\label{lem:eval}
	
	Fix $\gamma>\frac12$. Assume that $|V(x)|\leq M\la x\ra^{-2\gamma}$, $x\in\R^n$. Then, there exists a constant $C=C(n,m,\gamma, M)>0$ such that any $\lambda >C $ cannot be an eigenvalue of $H$.
	
\end{lemma}

\begin{proof}
	
	Assume that $\lambda>0$ is an eigenvalue of $H$, then there exists $0\neq\Psi\in L^2$ such that
	$$
		((-\Delta)^m+V)\Psi=\lambda \Psi, \quad \Rightarrow \quad ((-\Delta)^m-\lambda -i\epsilon)\Psi=-i\epsilon\Psi-V\Psi.
	$$
	Here $\epsilon>0$, so that $\lambda+i\epsilon$ is in the resolvent set of $H_0$.  Acting the resolvent operator $\mR_0(\lambda+i\epsilon)$ on both sides of the above expression yields
	$$
		\Psi=\mR_0(\lambda+i\epsilon)[-i\epsilon\Psi-V\Psi].
	$$
	Hence,
$$
		\|\Psi\|_{L^{2,-\gamma}} \leq \|\mR_0(\lambda+i\epsilon) \epsilon \Psi\|_{L^2}+ \|\mR_0(\lambda+i\epsilon) V \Psi\|_{L^{2,-\gamma}} 
		\les \bigg\| \frac{\epsilon}{|\xi|^{2m}-\lambda-i\epsilon} \hat{\Psi}\bigg\|_{L^{2}_{\xi}}+ \lambda^{\frac{1-2m}{2m}} \|V\Psi\|_{L^{2, \gamma}},
$$
	where we invoked the limiting absorption principle for the free operator on the second summand.  Taking $\epsilon \to 0^+$, the dominated convergence theorem suffices to conclude the first summand vanishes.  Using the bound on $V$, we have $\|V\Psi\|_{L^{2, \gamma}}\leq M\|\Psi\|_{L^{2, -\gamma}}$, hence
	\begin{align*}
		\|\Psi\|_{L^{2,-\gamma}} \les\lambda^{\frac{1-2m}{2m}}M\|\Psi\|_{L^{2, -\gamma}} \les \lambda^{\frac{1-2m}{2m}}  M \|\Psi\|_{L^{2, -\gamma}}.
	\end{align*}
	From here, we conclude  for large enough $\lambda$ that $\|\Psi\|_{L^{2,-\gamma}}=0$, and have that $\Psi=0$ a.e.  Thus $\lambda$ cannot be an eigenvalue of $H$.

\end{proof}

The above suffices to allow us to use the limiting absorption principle for the class of potentials we consider.
In the statement below $B(s,-s')$ is the space of bounded linear operators mapping $L^{2,s}\to L^{2,-s'}$, and assumes the lack of embedded eigenvalues in the continuous spectrum of $H$.  

\begin{theorem}[Theorem 3.9 in \cite{soffernew}]\label{thm:lap}
	
	For $k=0,1,2,3\dots,$ let $|V(x)|\les \la x \ra^{-\beta}$ for some $\beta>2+2k$, then for $s,s'>k+\frac12$, $\mR_V^{(k)}(z)\in B(s,-s')$ is continuous for $z>0$.  Furthermore,
	we have 
	$$
	\big\| \mR_V^{(k)}(z)
	\big\|_{L^{2,s}\to L^{2,-s'}}\les |z|^{\frac{1-2m}{2m}(1+k)}.
	$$
	
\end{theorem}
Note that, in particular, these bounds hold for the free resolvent, for which there are no embedded eigenvalues, as well as for the perturbed operator when $|z|$ is sufficiently large by Lemma~\ref{lem:eval}.
	
	\begin{proof}[Proof of Proposition~\ref{prop:hi tail}]
		
		The proof here mirrors closely the high energy argument in \cite{EGWaveOp}. Since we need only show boundedness, the argument is straight-forward.  We write $M=\ell_1+\ell_2$ where $\ell_1=\lfloor \frac{n}{4m}\rfloor +1$ is the number of iterations of the resolvent required to ensure $(\mR_0^\pm V)^{\ell_1}$ is locally $L^2$, and $\ell_2$ is selected large enough to ensure that there is sufficient decay in $\lambda$ using the limiting absorption principle in Theorem~\ref{thm:lap}.  This suffices to ensure the desired integral is bounded.

		Following the argument in Propositions~5.3 and 6.5 \cite{EGWaveOp}, we have the bound $\sigma>\f12$ and $\ell_1 =\lfloor \frac{n}{4m} \rfloor+1$ we have
		\begin{align}\label{eqn:iter res1}
			\|(V\mathcal R_0^+)^{\ell_1-1}V \mathcal R_0^\pm (\lambda^{2m})(\cdot, y)\|_{L^{2,\sigma}}  \les \frac{\lambda^{\ell_1(\frac{n+1}{2}-2m)}}{\la y\ra^{\frac{n-1}{2}}},
		\end{align}
	and similarly,
	\begin{align}\label{eqn:iter res2}
		\|\mathcal  (\mathcal R_0^+V)^{\ell_1 }(\lambda^{2m})(x, \cdot)\|_{L^{2,\sigma}}  \les \frac{\lambda^{\ell_1(\frac{n+1}{2}-2m)}}{\la x\ra^{\frac{n-1}{2}}}.
	\end{align}
	Selecting $L$ large enough so there are no eigenvalues on the support of the cut-off allow for iterated use of Theorem~\ref{thm:lap} which yields
	\begin{align}\label{eqn:lap rep}
		\|  (\mathcal R_0^+(\lambda^{2m}) V)^{\ell_2}\mathcal R_V^+(\lambda^{2m})(V \mathcal R_0^+(\lambda^{2m}))^{\ell_2}\|_{L^{2,\sigma}\to L^{2,-\sigma}} \les \lambda^{(2\ell_2+1)(1-2m)}.
	\end{align}
	Combing these, we arrive at the bound
	\begin{multline}\label{eqn:hi tail int}
		\bigg|\int_0^\infty e^{it\lambda^{2m}}\psi(\lambda^{2m}/L^{2m}) \lambda^{(n+2)m-n-1}(\mR_0^\pm(\lambda^{2m}) V)^M \mR_V^\pm (\lambda^{2m}) (V\mR_0^\pm(\lambda^{2m}))^M \, d\lambda \bigg| \\ 
		\les  \frac{1}{ (\la x\ra \la y \ra)^{\frac{n-1}{2}} } \int_0^\infty \psi(\lambda^{2m}/L^{2m}) \lambda^{((n+2)m-n-1+2\ell_1 (n+1-4m)+(2\ell_2+1)(1-2m) )}\, d\lambda.
	\end{multline}
	Having selected $\ell_1=\lfloor \frac{n}{4m}\rfloor +1$ we selected $\ell_2\in \mathbb N$ so that
	$$
		\ell_2>\frac{1+n(m-1)+2\ell_1 (n+1-4m) }{4m-2}.
	$$	
	Thus, we may bound
	\begin{align*}
		|\eqref{eqn:hi tail int}| \les \frac{1}{ (\la x\ra \la y \ra)^{\frac{n-1}{2}} } \int_0^\infty \psi(\lambda^{2m}/L^{2m}) \lambda^{-2}\, d\lambda \les \int_{\lambda \approx L} \lambda^{-2}\, d\lambda \les L^{-1}.
	\end{align*}
	This is uniformly bounded in $x,y\in\R^n$ and $L>1$.
		
	\end{proof} 

	We note that these bounds can be shown to hold for a much larger class of potentials $V$, and one can show that the tail decays in $t$ by utilizing the oscillation in the integral.  Such bounds are interesting, but are not needed for our purpose here.

\subsection{Failure of the dispersive bound for the full evolution}

We now prove that the full evolution cannot satisfy the dispersive bound, and consequently that the wave operators are not bounded on $L^\infty$.  The proof follows that in \cite{GV} for the case of $m=1$, though the frequency localization allows us to avoid many technical issues such as regularity of the threshold energies or considering larger classes of potentials.  Namely we show the following.

\begin{prop}\label{prop:nondisp}
	
	Suppose that $n>3$ and $n>4m-1$, $\psi$ is a smooth cut-off to $\lambda \approx 1$ and $0\leq \alpha<\frac{n+1}{2}-2m$.  There cannot exist a bound of the form
	$$
		\sup_{t>0, L>0} t^{\frac{n}2}\| H^{\frac{n(m-1)}{2m}} e^{itH} \psi(H/L^{2m})  P_{ac}(H)f\|_\infty \leq C(V)  \|f\|_1
	$$
	with $C(V)<\infty$ for all $V\in C^{\alpha}(B(0,2))$.
	
\end{prop}

\begin{proof}
	
	Assume such a bound can hold.  Write $V=\theta W$ with $\theta\in [0,1]$ and $W\in C^{\alpha}(B(0,2))$.  The assumed dispersive bound implies the following holds uniformly in $L>C_V$ where $C_V$ is the constant in the proof of Proposition~\ref{prop:hi tail} chosen to ensure there are no embedded eigenvalues on the support of the cut-off.  Taking $L>C_V$ suffices to ensure the argument for the tail holds for any choice of $0\leq \theta \leq 1$.
	\begin{multline*}
		\sup_{0<t<1, L>C_V} t^{\frac{n}{2}}|\la  H^{\frac{n(m-1)}{2m}} e^{itH} \psi(H/L^{2m})  P_{ac}(H)f, g\ra | \\
		= \sup_{0<t<1, L>C_V} \frac{m \,t^{\frac{n}{2}}}{\pi } \bigg|\int_0^\infty e^{it\lambda^{2m}} \lambda^{2m-1+n(m-1)} \psi(\lambda^{2m}/L^{2m}) \la [\mR_{\theta W}^+(\lambda^{2m})-\mR_{\theta W}^-(\lambda^{2m})]f,g\ra \, d\lambda 		\bigg|\\
		\leq C(\theta, W)  \|f\|_1\|g\|_1.
	\end{multline*}
	We take both $f$ and $g$ to be the function $f_L$ defined above \eqref{UBBalpha}.  Expanding the perturbed resolvent into a Born series as in \eqref{eqn:born identity} allows us to express the evolution as a sum a polynomial in $\theta$ of degree $2M$ and the tail of the Born series.  The coefficients of this polynomial depend on $t$, $W$, and $L$, are defined by
	\begin{multline}
		a_k(W,L)=\frac{m\, t^{\frac{n}{2}} }{\pi i} \int_0^\infty e^{it\lambda^{2m}} \lambda^{2m-1+n(m-1)} \psi(\lambda^{2m}/L^{2m})\\ \la [\mR_{0}^+(\lambda^{2m})[W\mR_0^+(\lambda^{2m})]^k-\mR_{0}^-(\lambda^{2m})[W\mR_0^-(\lambda^{2m})]^k]f_L,f_L\ra \, d\lambda 	
	\end{multline}
	
	Proposition~\ref{prop:hi tail} shows that the tail obeys the desired bound provided $L$ is sufficiently large.  It follows that the Born series terms, the polynomial in $\theta$, must also obey the bound as well.  Writing
	$$
		P_{L,t}(\theta)=\sum_{k=0}^{2M}   a_k(W,L)\theta^k
	$$
	By assumption, $\sup_{0<t<1, L>C_V} |P_{L,t}(\theta)|$ is finite for each $0\leq \theta\leq 1$.  Hence the maximum $0\leq j\leq 2M$ of $\sup_{0<t<1, L>C_V}|P_{L,t}(\frac{j}{2M})|$ is bounded.  The value of the polynomial at $2M+1$ points suffices to solve for the values of the coefficients, and allows us to conclude that each of the coefficients is bounded.  Hence, we conclude that
	$$
		\sup_{0<t<1, L>C_V}t^{\frac{n}{2}}|a_1(W,L)|\leq C(W)<\infty.
	$$
	This, however, is false in light of the estimate in subsection~\ref{sec:first}.  Hence, the assumption is false and such a bound cannot hold.

\end{proof}

	The intertwining identity \eqref{eq:intertwining} quickly establishes the following corollary, which is the $p=\infty$ statement in Theorem~\ref{thm:main}.

\begin{corollary}
	
	Suppose that $n>3$ and $n>2m$. Then for any $\alpha<\frac{n+1}{2}-2m$, there exists a real-valued compactly supported potential $V$ in $C^{\alpha}(\R^n)$   for which the wave operators $W^\pm$ do not extend to bounded operators on $L^\infty(\R^n)$.

\end{corollary}

\section{Extension to $L^p$}\label{sec:ext to p}

In this section we complete the proof of Theorem~\ref{thm:main} to consider the full range of $p$ on which our counterexample to $L^p$ boundedness applies.
The argument presented in the previous section can be adapted to show the failure of dispersive bounds from $L^p$ to $L^{p'}$ for $1\leq p\leq 2$, which implies unboundedness of the wave operators on $L^{p'}$.  With $\psi$ the same cut-off to frequencies $\lambda \approx 1$, by Corollary~\ref{cor:free} the free evolution satisfies the bound
$$
	\left\|H_0^{\frac{n(m-1)}{2m}(\frac{2}{p}-1)}e^{itH_0}\psi(H_0/L^{2m})\right\|_{p\to p'} \les |t|^{\frac{n}{2}(1-\frac{2}{p})}.
$$
Here we show that the dispersive bound for the perturbed operator fails by showing that one cannot have a bound of the form
\begin{align}\label{eqn:Lp disp}
	\sup_{L>0,t>0}t^{\frac{n}{2}(1-\frac{2}{p})} \bigg|\left\la H^{\frac{n(m-1)}{2m}(\frac{2}{p}-1)}e^{itH}\psi(H/L^{2m})f, g \right \ra \bigg| \leq C(V)
\end{align}
for $f,g$ unit vectors in $L^p$.  As in the previous section, we'll select $f=g$ to be functions that concentrate at zero as $L\to \infty$.  Namely, we'll use $f_{L,p}(x)=(CL)^{\frac{n}p} f(x(CL))$ where $f$ is the $L^p$-normalized characteristice function of $B(0,1)$.  The argument proceeds analogously, we show that the first term of the Born series grows too fast as $t\to 0^{+}$ and the tail remains bounded.  

For the first term of the Born series, we note that we now consider the difference of kernels
\begin{multline*}
	K^{\pm}_{L,p,t}(x,y)\\
	= \int_0^\infty \int_{\R^n} e^{it\lambda^{2m}} \lambda^{2m-1} \lambda^{n(m-1)(\frac{2}{p}-1)} \mR_0^\pm(\lambda^{2m})(x,z) V(z) 
	\mR_0^\pm(\lambda^{2m})(z,y) \psi(\lambda^{2m}/L^{2m}) dz\, d\lambda .
\end{multline*} 
The oscillatory integral argument applies ver batim, there are three changes.  Due to the different power of $\lambda$, the $\lambda \mapsto \lambda L$ rescaling produces a power of $n(m-1)(\frac{2}{p}-1)$ on $L$.  The time factor in front of the integral $t^{\frac{n}{2}(1-\frac{2}{p})}$ becomes $L^{nm-\frac{n}{2}+\frac{n}p-\frac{2mn}p}$ using the scaling $t=L^{-(2m-1)}$.  The final change comes from the fact that the positive functions $f_{L,p}$ are integrated directly, we have
$\|f_{L,p}\|_1=(CL)^{n/p-n}\|f\|_1\approx L^{n/p-n}\|f\|_p$ since $f$ is a normalized characteristic function of $B(0,1)$.  

Hence we consider the linear operator 
$$a_{1,p,L}(V)= L^{nm-\frac{n}{2}+\frac{n}p-\frac{2mn}p} \int_{\R^n\times\R^n} K_{L,L^{-(2m-1)}}(x,y) f_{L,p}(x)f_{L,p}(y) dx dy.
$$
With the rescaling accounted for, the stationary phase arguments apply directly to show that there is a $C_p\in \mathbb C$ so that
$$
\Re \big[C_p K_{L,p}^+(x,y)-C_p K_{L,p}^-(x,y)\big] \gtrsim  L^{\frac{n}{p}+\frac{1-n}{2}-2m}\delta,
$$
for $\delta\ll 1$ fixed, and for $L$ sufficiently large and $|x|,|y|\ll \frac1L$.  Therefore, we have 
$$\frac{|a_{1,p,L}(V_L)|}{\|V_L\|_{C^\alpha(B(0,2))}}\gtrsim L^{\frac{n}{p}+\frac{1-n}{2}-2m-\alpha }\delta\to \infty
$$
as $L\to\infty$ unless $\alpha\geq \frac{n}{p}+\frac{1-n}{2}-2m$, provided $p<\frac{2n}{4m+n-1}$ to ensure the right side of the inequality is positive.  

For the tail of the Born series, the extension is straight-forward.
By selecting $M$ appropriately large depending on the parameters $j,k$ below, the proof of Proposition~\ref{prop:hi tail} can be adapted to show
\begin{align*}
	\sup_{x,y\in\R^n} \bigg|\int_0^\infty e^{it\lambda}\psi(\lambda/L^{2m}) \lambda^{k}(\mR_0^\pm(\lambda) V)^M \mR_V^\pm (\lambda) (V\mR_0^\pm(\lambda))^M(x,y) \, d\lambda \bigg| \les L^{-j}, 
\end{align*}
for any choice of $j,k\in \mathbb N$.  Again using the $L^p$-normalized  function $f_{L,p}$ we have
$\|f_{L,p}\|_1\approx L^{n/p-n}\|f\|_p$.  Hence, we have
\begin{multline*}
	\bigg|\int_0^\infty e^{it\lambda}\psi(\lambda/L^{2m}) \lambda^{k}\la (\mR_0^\pm(\lambda) V)^M \mR_V^\pm (\lambda) (V\mR_0^\pm(\lambda))^M(x,y)f_{L,p}, f_{L,p}\ra \, d\lambda \bigg|\\  
	\les \sup_{x,y\in\R^n} \bigg|\int_0^\infty e^{it\lambda}\psi(\lambda/L^{2m}) \lambda^{k}(\mR_0^\pm(\lambda) V)^M \mR_V^\pm (\lambda) (V\mR_0^\pm(\lambda))^M(x,y) \, d\lambda \bigg| \|f_{L,p}\|_1^2
	 \les L^{-1}.
\end{multline*}
uniformly in $L>C_{n,m,V}$, where $C_{n,m,V}$ is the constant chosen so that there are no embedded eigenvalues on the support of the cut-off.

The polynomial argument now applies with coefficients defined by
\begin{multline}
	a_k(W,L,p)=\frac{m\, t^{\frac{n}{2} (\frac{2}{p}-1)} } {\pi i} \int_0^\infty e^{it\lambda^{2m}} \lambda^{2m-1+n(m-1)(\frac{2}{p}-1) } \psi(\lambda^{2m}/L^{2m})\\ \la [\mR_{0}^+(\lambda^{2m})[-W\mR_0^+(\lambda^{2m})]^k-\mR_{0}^-(\lambda^{2m})[-W\mR_0^-(\lambda^{2m})]^k]f_{L,p},f_{L,p}\ra \, d\lambda .	
\end{multline}
Hence, we are able to conclude that the $L^p\to L^{p'}$ dispersive bound in \eqref{eqn:Lp disp} fails to hold for all $V\in C^{\alpha}(B(0,2))$ if $0\leq \alpha< \frac{n}{p}+\frac{1-n}{2}-2m$.  By appealing to the intertwining identity \eqref{eq:intertwining}, this shows that the wave operators are unbounded on  $L^q$ for $\frac{2n}{n-4m+1}< q\leq \infty$.

A consequence of these bounds are results that are, to the best of the authors' knowledge, new for the classical $m=1$ Schr\"odinger evolution.  Namely, with $p'$ the H\"older conjugate of $p$ we have the following.

\begin{corollary}\label{cor:m=1}
	
	Suppose that $n>3$ and   $1\leq p<\frac{2n}{n+3}$. Then for any $0\leq \alpha< \frac{n}{p}-\frac{n+3}{2}$ there exists a real-valued compactly supported potential $V$ in $C^{\alpha}(\R^n)$  for which the dispersive bound  
	$$
		\bigg\| e^{it(-\Delta+V)}P_{ac}(-\Delta+V)\bigg\|_{p\to p'} \les |t|^{-\frac{n}{2}(\frac2p-1)}
	$$
	fails. 

In particular, the wave operators $W^\pm$ do not extend to bounded operators on $L^{p'}(\R^n)$.  Consequently, we conclude that for all $\frac{2n}{n-3}<q\leq \infty$ there exist a compactly supported continuous potential for which the wave operators are unbounded on $L^q(\R^n)$.

\end{corollary}
\begin{proof}
Assuming that the $L^p\to L^{p\prime}$  bound in the corollary holds and using Theorem 2.1 in \cite{JenNak} for the uniform $L^p$ boundedness of $\psi(H/L)$ we conclude that the bound \eqref{eqn:Lp disp} should hold, which is a contradiction.
\end{proof}

\end{document}